\documentclass[a4paper,12pt]{amsart}
\usepackage{amsthm}
\usepackage{amsfonts}
\usepackage{amssymb}
\usepackage{graphicx}
\usepackage{texdraw}
\usepackage{psfrag}

\newtheorem{thm}{Theorem}

\newtheorem{lemma}{Lemma}

\begin{document}

\title{A family of non-isomorphism results}
\author{Collin Bleak}
\address{ Department of Mathematics\\
University of Nebraska\\ 
Lincoln, NE 68588-0130, USA\\}
\email{cbleak2@math.unl.edu}
\author{Daniel Lanoue}
\address{Department of Mathematics\\
Northeastern University\\
Boston, MA 02115, USA\\
}
\email{lanoue.d@neu.edu}

\keywords{Higher Dimensional R. Thompson Groups nV, germs, Rubin's Theorem}
\subjclass{20E32; 20B22; 20F65; 20F38}

\bibliography{ref}
\bibliographystyle{amsplain}

\date{\today}
\maketitle

\begin{abstract}
We give a short argument showing that if $m$, $n\in\left\{1,2,\ldots\right\}\cup
\left\{\omega\right\}$, then the groups $mV$ and $nV$ are not isomorphic.
This answers a question of Brin.
\end{abstract}

\section{Introduction}

In the paper \cite{BrinHigherV}, Brin introduces groups $nV$, where
$n\in \left\{1,2,\ldots\right\}\cup\left\{\omega\right\}$.  These
groups can be thought of as ``higher dimensional'' analogues of
Thompson's group $V$.  In particular, while $V = 1V$ is a group of
homeomorphisms of the standard, deleted-middle-thirds Cantor set $C$,
the group $nV$ can be thought of as a group of homeomorphisms of
$C^n$.  

The groups $nV$ are all simple (see \cite{BrinTalk}) and finitely
presented (\cite{BleakHennigMatucci}.  Also, they each contain copies
of every finite group.  Because of these properties, it is non-trivial
to detect whether these groups are pairwise isomorphic through the use
of standard algebraic machinery.  On the other hand, by Rubin's Theorem, much of the structure of these groups is encoded in their actions as groups of homeomorphisms.  Thus, the authors of this note turned to dynamical/topological approaches.

In the seminal article \cite{BrinHigherV}, Brin shows that the group
$2V$ is not isomorphic to $V$.  He also asks (in his talk \cite{BrinTalk}) whether $mV$ and $nV$ can be
isomorphic if $m\neq n$.  In this paper, we answer that question.

As Brin demonstrates in \cite{BrinHigherV}, the groups $V$ and $2V$
act on $C$ and $C^2$ in such a way that Rubin's Theorem (see
\cite{Rubin}) applies; any isomorphism $\phi:V \to 2V$ would induce a
unique homeomorphism $\psi:C \to C^2$, so that for $v\in V$,
$\phi(v) = \psi\circ v\circ\psi^{-1}$.  As $C$ and $C^2$ are
homeomorphic, Rubin's Theorem does not directly demonstrate that the
groups $V$ and $2V$ are non-isomorphic.  However, in
\cite{BrinHigherV}, Brin in finds an element in $2V$ which acts with
periodic orbits of length $k$ for arbitrarily large integers $k$.  As
he also demonstrates that $V$ does not contain such an element, he is
able to deduce that $1V = V$ is not isomorphic with $2V$.

Although we also employ Rubin's Theorem, our approach is different
from that used by Brin above (although it owes somewhat to Brin's work
in \cite{BrinChameleon}, where he classifies the automorphism groups
of R. Thompson's groups $F$ and $T$).  Our chief result is stated
below.

\begin{thm}
\label{nonIso}
Let $m$, $n\in\left\{1,2,\ldots\right\}\cup \left\{\omega\right\}$, with $m\neq n$, then $mV$ is not isomorphic to $nV$.
\end{thm}

\section{Background Requirements}
In this paper, most of the work is in setting up the correct perspective.  Once that is achieved, the proof of Theorem \ref{nonIso} is almost a triviality.

Throughout this section, we will provide the background and set the
stage for that proof.

\subsection{A group of germs}
Suppose $X$ is a topological space, and $H$ is a subgroup of the full
group of homeomorphisms of $X$.  For any $x\in X$, define the set
$Fix_{(H,x)} = \left\{h\in H\,|\, h(x) = x\right\}$.

We place an equivalence relation on $Fix_{(H,x)}$.  Say that $f$,
$g\in Fix_{(H,x)}$ are equivalent if there is a neighborhood $U$ of
$x$ so that $f|_U=g|_U$.  In this case we write $f\sim_x g$ and we
denote the equivalence class of any element $h\in Fix_{(H,x)}$ by
$[h]_x$.  We further denote the set of equivalence classes that result
by $G_{(H,x)}$.  This last set is known as the set of germs of $H$ at
$x$ which fix $x$.

The following is standard. It is also given explicitly as a portion of
Lemma 3.4 in \cite{BrinChameleon}.

\begin{lemma}
Given a topological space $X$, a point $x\in X$, and a group of
homeomorphisms $H$ of $X$, the set $G_{(H,x)}$ forms a group under the
operation $[f]_x*[g]_x = [fg]_x$.
\end{lemma}

\subsection{$C^n$, $n$-rectangles, and the groups $nV$}
There is a well-known correspondence between finite binary strings,
and subsets of the Cantor set $C$.  We interpret a string as inductive
choices of halves.  For instance, the string ``011'' would means take
the left half of the Cantor set (throwing out the right half), then take
the right half of what remains, and finally pass to the right half of
that.  The diagram below illustrates this notation.

\begin{center}
\psfrag{011}[c]{$011$}
\includegraphics[height=125pt,width = 200 pt]{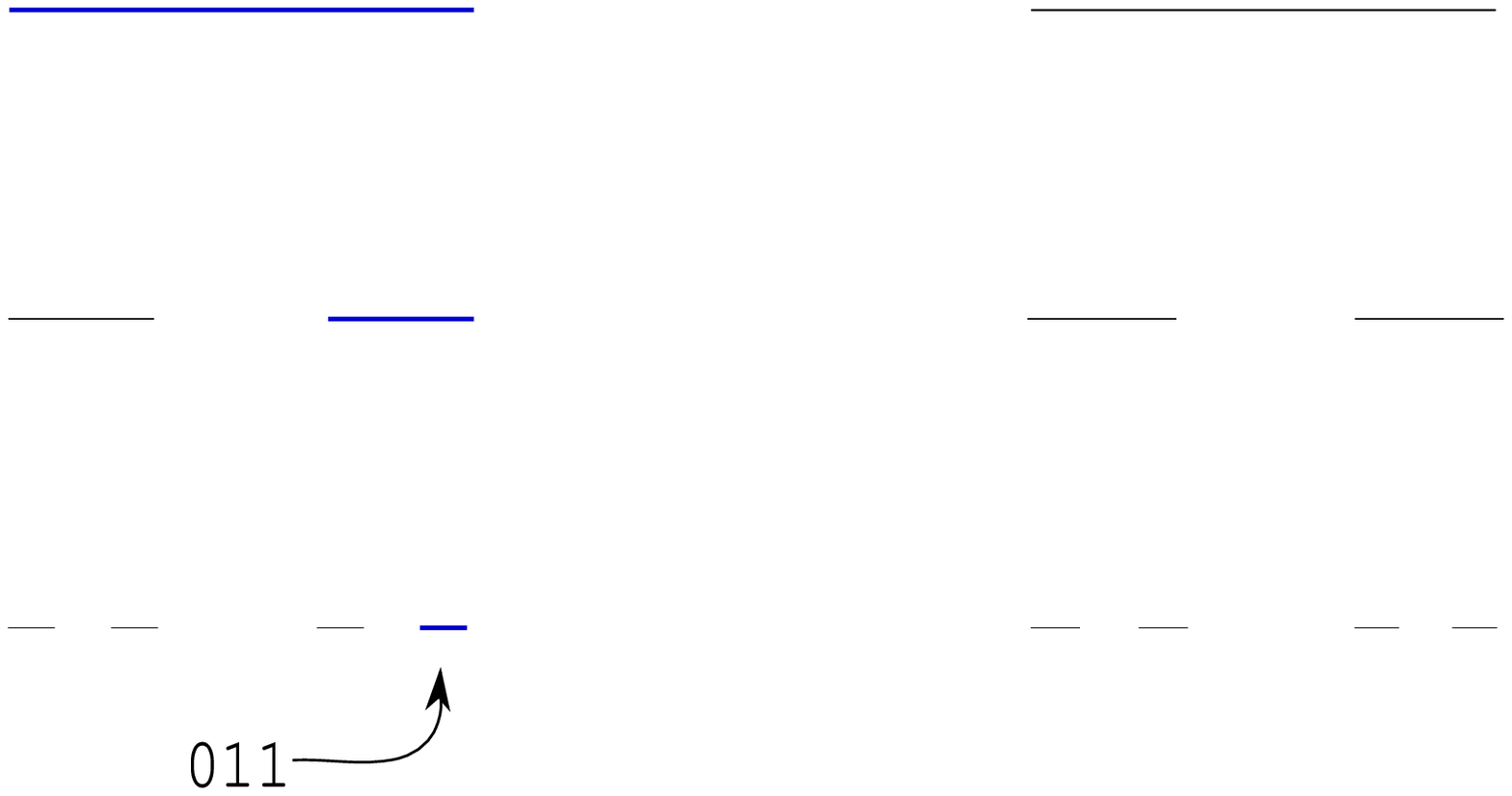}
\end{center}

If we pass to the limit and consider infinite binary strings, we
obtain the standard bijection from $2^N$ to the Cantor set.  That is, we abuse notation by having
an element $s$ of the Cantor set correspond to a map $s:N\to
\left\{0,1\right\}$. In this notation, we would denote $s(0)$ by
$s_0$, so that, by an abuse of notation, we can write $s$ as an
infinite string, that is $s = s_0s_1\ldots$\,. In this case, we can
define a \emph{prefix of $s$} as any finite substring beginning with
$s_0$.  We will call an infinite binary string $w$ an \emph{infinite
tail of $s$} if $s = Pw$ for some finite prefix $P$ of $s$.

  Elements $s, t \in 2^{N}$ are near to each other when they have long
  common prefixes. This induces the standard topology of the Cantor
  set.  In this description, a point $z$ of the Cantor set will be
  rational if and only if it corresponds to an infinite string of the
  form $P\overline{w} = Pww\ldots$, where $P$ is a prefix string and
  $w$ is some non-empty finite string.

We will now fix $n$ as a positive integer.  Let us first define a
special class of subsets of $C^n$.  $R$ is an \emph{$n$-rectangle in
$C^n$} if there is a collection of finite binary strings $P_0$, $P_1$,
$\dots$, $P_{n-1}$ so that\\ \\$R = \left\{z\in C^n\,|\, z =
(x_0,x_1,x_2,\ldots,x_{n-1}) \textrm{ and } x_i = P_iz_i \textrm{
where } z_i\in 2^N\right\}.$\\ \\In this case, by an abuse of notation,
we will say $R=(P_0, P_1,\ldots,P_{n-1})$.

We now define a special class of maps, $n$-rectangle maps, as follows.
Suppose $D = (P_0,P_1,\ldots,P_{n-1})$ and $R=(Q_0,Q_1,\dots,
Q_{n-1})$ are $n$-rectangles.  We define the \emph{$n$-rectangle map}
$\tau_{(D,R)}:D\to R$ by the rule $z\mapsto z'$ where the $i$'th
coordinate $P_iz_i$ of $z$ determines the $i$'th coordinate $Q_iz_i$
of $z'$. We may also refer to these as prefix maps.  

We note that a prefix map $\tau_{(D,R)}$ (as above) scales dimension
$i$ in accord with the length of $P_i$ and the length of $Q_i$.  For
instance, if $P_i = 1$ and $Q_i = 011$, then in dimension $i$, the map
will take a half of the Cantor set and map it affinely over an eighth of
the Cantor set.  In particular, any such map will have a scaling
factor of $2^n$ for $n\in Z$.  (In terms of metric scaling, perceiving
the Cantor set as the standard deleted-middle-thirds subset of the interval,
the scaling factors would of course be $3^n$.  We will use the $2^n$
point of view in the remainder of this note.)

We can now define a \emph{pattern} to be a partition of $C^n$ into a
finite collection of $n$-rectangles.  An element $f$ of $nV$ now
corresponds to a homeomorphism from $C^n$ to $C^n$ for which there is
an integer $k$, a domain pattern $D$, and a range pattern $R$, each
pattern with $k$ rectangles, so that $f$ can be realized as the union
of $k$ disjoint $n$-rectangle maps, each carrying a $n$-rectangle of
$D$ to a $n$-rectangle of $R$. It should be noted that two different pairs
of partitions can correspond to the same map. The diagram below demonstrates a
typical such map in $2V$.

\begin{center}
\psfrag{1}[c]{$\,\,1$}
\psfrag{2}[c]{$\,\,2$}
\psfrag{3}[c]{$\,\,3$}
\includegraphics[height=140pt,width = 200 pt]{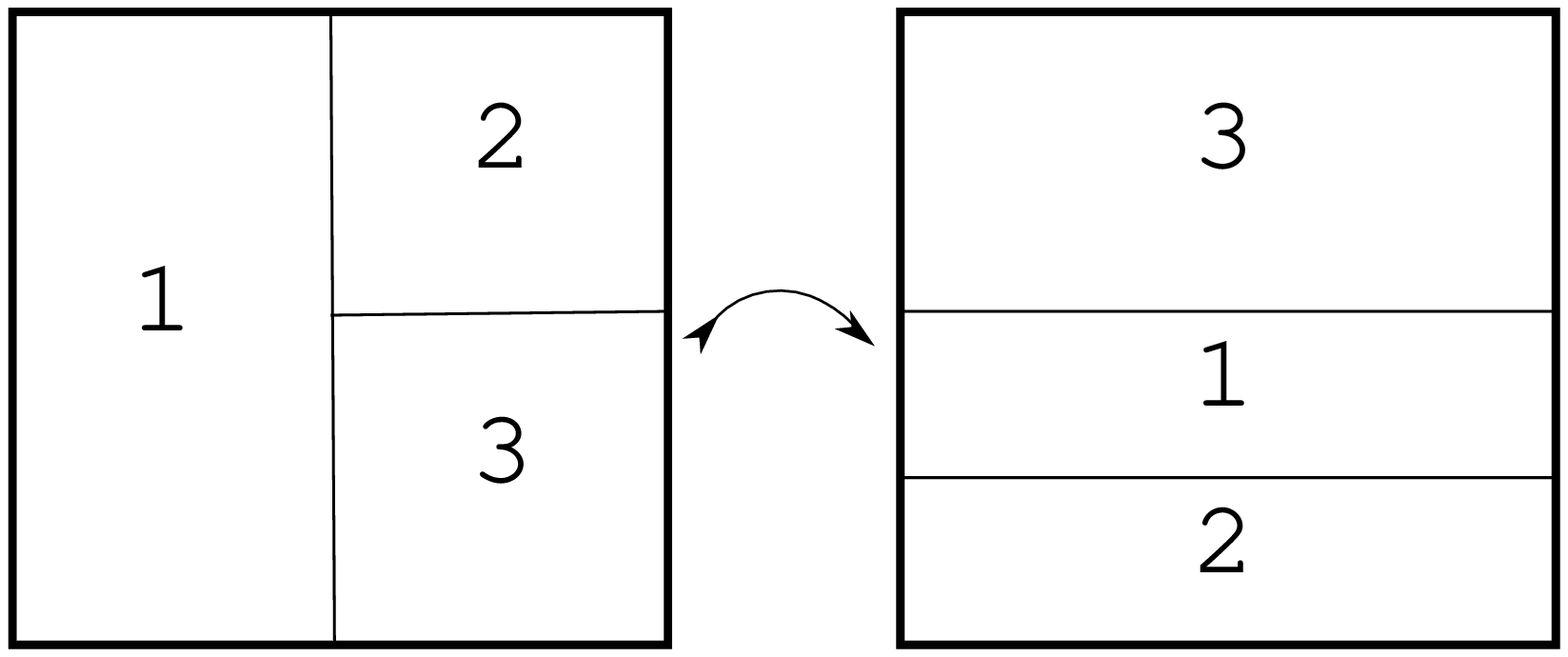}
\end{center}

In the case of $\omega V$, the $\omega$-rectangles are only allowed to
restrict the domain in finitely many dimensions.  That is, $R =
(P_0,P_1,P_2,\ldots)$ is an $\omega$ rectangle if and only if only
finitely many of the $P_i$ are non-empty strings.  Thus, $\omega V$ can be
thought of as a direct union of the $nV$ groups for finite $n$.

\subsection{Rubin's Theorem and some groups of germs}
In order to state Theorem \ref{RubinThm} below, we need to give a
further definition.  If $X$ is a topological space and and $F$ is a
subgroup of the group of homeomorphisms of $X$, then we will say that
$F$ is \emph{locally dense} if and only if for any $x\in X$ and open
neighborhood $U$ of $x$ the set
\[
\left\{f(x)|f\in F, f|_{(X-U)} = 1|_{(X-U)}\right\} 
\]
 has closure containing an open set.

The following is the statement of Rubin's Theorem, as
given by Brin as Theorem 2 in \cite{BrinHigherV}.  It is a
modification of Rubin's statement Theorem 3.1 in \cite{Rubin}, which
statement appears to contain a minor technical error.

\begin{thm}[Rubin]
\label{RubinThm} 
Let $X$ and $Y$ be locally compact, Hausdorff topological spaces
without isolated points, let H($X$) and H($Y$) be the automorphism
groups of X and Y, respectively, and let $G \subseteq$ H($X$) and $H
\subseteq$ H($Y$) be subgroups. If $G$ and $H$ are isomorphic and are
both locally dense, then for each isomorphism $\phi:G \to H$ there is
a unique homemorphism $\psi:X \to Y$ so that for each $g \in G$, we
have $\phi(g) = \psi g \psi^{-1}$.
\end{thm}

If we combine Rubin's Theorem with our previous work on the group of
germs, we get a lemma which appears to be a very mild extension of
Lemma 3.4 from \cite{BrinChameleon}.
\begin{lemma}
\label{inducedGermIso}
Suppose $X$ and $Y$ are locally compact, Hausdorff topological spaces
without isolated points, and that $G$ and $H$ are respectively
subgroups of the homeomorphism groups of $X$ and $Y$, so that $G$ and
$H$ are both locally dense.  Suppose further that $\phi:G\to H$ is an
isomorphism, and that $\psi$ is the homeomorphism induced by Rubin's
Theorem.  If $x\in X$ and $y\in Y$ so that $\psi(x) = y$, then $\psi$
induces an isomorphism $\overline{\psi}:G_{(G,x)}\to G_{(H,y)}$.
\end{lemma}
\begin{proof}
This is a straightforward exercise in calculation.

For $[v]_x\in G_{(G,x)}$, define $\overline{\psi}([v]_x) = [\psi v \psi^{-1}]_y$.

We first show that  $\overline{\psi}$ is well defined.
Let $f \sim_x g$ with $f, g \in Fix_{(G,x)}$, so that $f|_U = g|_U$ with $U$ an open neighborhood of $x$. 
Then $N = \psi(U)$ is an open neighborhood of $y = \psi(x)$ since $\psi$ is a homeomorphism.
 For $z \in N$, $\psi^{-1}(z) \in U$, so $\psi f \psi^{-1}(z) = \psi g \psi^{-1}(z)$. Thus 
$[\psi f \psi^{-1}]_y = [\psi g \psi^{-1}]_y$.

Now we show that $\overline{\psi}$ is a homomorphism.
We have
\[
\overline{\psi}([f]_x[g]_x) = \overline{\psi}([fg]_x) = [\psi f g \psi^{-1}]_y = [\psi f \psi^{-1} \psi g 
\psi^{-1}]_y 
\]
\[= [\psi f \psi^{-1}]_y [\psi g \psi^{-1}]_y = \overline{\psi}([f]_x) \overline{\psi}([g]_x).
\]
 
Finally, we show that $\overline{\psi}$ is a bijection.  Since
$\psi^{-1}$ conjugates $H$ to $G$, we also have an induced map
$\overline{\psi^{-1}}$.  Now, direct calculation as above shows that
$\overline{\psi^{-1}}\circ\overline{\psi}$ and
$\overline{\psi}\circ\overline{\psi^{-1}}$ are the identity maps on
$G_{(G,x)}$ and $G_{(H,y)}$ respectively.

\end{proof}

\section{Conclusion}
We are now in a position to prove Theorem \ref{nonIso}. We describe
our results using notation for some $n < \omega$ and generally leave to the
reader the extension to $n = \omega$.

\subsection{Some germs of $nV$}
Let us calculate the group $G_{(nV,x)}$, for $x \in C^n$.

For $x = (x_1,\ldots\, ,x_n) \in C^n$, let $|x|$ denote the cardinality of the set $\{i : x_i$ rational$\}$.

We now observe the following.

\begin{lemma}
For $x \in C^n$, $G_{(nV,x)} \cong Z^{|x|}$.
\end{lemma}

\begin{proof} Let $x = (x_0, x_1, \ldots, x_{n-1})$, with $|x| = k$ and assume without meaningful loss of generality that $x_0, .. x_{k-1}$ are rational.
In particular, let us write $x_i= A_i\overline{w_i}$ for each index
$i$ with $0\leq i \leq k-1$, where $A_i$ is the shortest prefix so
that $x_i$ can be written in this fashion, and where $w_i$ is the
shortest word so that an infinite tail of $x_i$ is of the form
$\overline{w_i}$. Suppose $f \in nV$ with $f(x) = x$, and let $j$ be
the minimal integer so that $f$ admits a decomposition as a union of
$j$ disjoint $n$-rectangle maps, $\left\{f_i:D_i\to R_i\,|\, 1\leq
i\leq j\right\}$, with $f = f_1\cup f_2\ldots \cup f_j$.

Assume further that $a$ is the index so that $x\in D_a$.  By our
earlier assumptions, $x\in R_a$ as well.  We can now restrict our
attention to $\tau_{(D_a,Q_a)}$.  We have that there are finite
strings $P_{(a,i)}$ and $Q_{(a,i)}$ so that $R_a = (P_{(a,0)},P_{(a,1)},\ldots,
P_{(a,n-1)})$ and $Q_a = (Q_{(a,0)},Q_{(a,1)},\ldots ,Q_{(a,n-1)})$.

Since $x$ is fixed by $f$, we must have that for each index $m\geq k$, $P_{(a,m)} = Q_{(a,m)}$.  Also, we see that for each index $m <k$, $P_{(a,m)} = A_m(w)^{s_m}$ and $Q_{(a,m)} = A_m(w)^{t_m}$, for non-negative integers $s_m$ and $t_m$.  Define for $f$,  the tuple $(s_0-t_0, s_1-t_1, \ldots , s_{k-1} - t_{k-1})\in Z^k$.  We will call this association $\hat{\sigma}_x:Fix_{(nV,x)}\to Z^k$.

If $g$, $h\in Fix_{(nV,x)}$, with
$[g]_x = [h]_x$, then $\hat{\sigma}_x(g) =
\hat{\sigma}_x(h)$, as otherwise the scaling of the two maps in
some dimension could not be the same in some small neighborhood of $x$.

In particular, the map $\hat{\sigma}_x$ induces a set map $\sigma_x:G_{(nV,x)}\to Z^k$.  The reader may now confirm that the map $\sigma_x$ is in fact an isomorphism of groups.
\end{proof}

\subsection{Germ isomorphisms}
For any integer $n$, $C^n$ is a locally compact, Hausdorff topological
space without isolated points (recall that $C^n$ is homeomorphic to
$C$ for any integer $n$).  Given any $x\in C^n$ with open neighborhood
$U$, and $y\in U$, there is a map in $nV$ which swaps a very small
$n$-rectangle around $x$ with a very small $n$-rectangle around $y$,
and is otherwise the identity.  In particular, the group $nV$ acts
locally densely on $C^n$.  A similar argument works for $\omega V$,
where our $n$-rectangles are given with long prefixes for the first
$N$ coordinates for some $N$, and are empty afterwords.  Thus, Rubin's
Theorem applies for $nV$ for any index $n$ in
$\left\{1,2,\ldots\right\}\cup\left\{\omega\right\}$.

We are finally in position to prove our main theorem.

Suppose $m$ and $n$ are valid indices, with $m<n$, and suppose
$\phi:nV \to mV$ is an isomorphism, with $\psi:C^n\to C^m$ the
homeomorphism induced by Rubin's Theorem.  Let $x$ be a point in $C^n$
with all $n$ coordinates rational, and let $y = \psi(x)$.  The map
$\psi$ induces an isomorphism
$\overline{\psi}:G_{(nV,x)}\to G_{(mV,y)}$.  But observe, the rank of
$G_{(nV,x)}$ is $n$, while the rank of $G_{(mV,y)}$ must be less than
or equal to $m$.  In particular, no such isomorphism $\overline{\psi}$ can
exist.

\end{document}